\documentclass[12pt]{article}
\usepackage{e-jc}

\usepackage{amsmath,amssymb}
\usepackage{graphicx}
\usepackage[colorlinks=true,citecolor=black,linkcolor=black,urlcolor=blue]{hyperref}

\usepackage{mathtools}

\dateline{Oct 17, 2019}{Jan 23, 2020}{Feb 7, 2020}


\Copyright{The authors. Released under the CC BY license (International 4.0).}

\begin{document}
\title{Multicolor Ramsey numbers via pseudorandom \\ graphs}
\author{Xiaoyu He\thanks{Supported by NSF Graduate Research Fellowship DGE-1656518.}\\
\small Department of Mathematics\\[-0.8ex]
\small Stanford University\\[-0.8ex] 
\small Stanford, U.S.A.\\
\small\tt alkjash@stanford.edu\\
\and
Yuval Wigderson\thanks{Supported by NSF Graduate Research Fellowship DGE-1656518.}\\
\small Department of Mathematics\\[-0.8ex]
\small Stanford University\\[-0.8ex] 
\small Stanford, U.S.A.\\
\small\tt yuvalwig@stanford.edu}

\maketitle
\begin{abstract}
A weakly optimal $K_s$-free $(n,d,\lambda)$-graph is a $d$-regular $K_s$-free graph on $n$ vertices with  $d=\Theta(n^{1-\alpha})$ and spectral expansion $\lambda=\Theta(n^{1-(s-1)\alpha})$, for some fixed $\alpha>0$. Such a graph is called optimal if additionally $\alpha = \frac{1}{2s-3}$. We prove that if $s_{1},\ldots,s_{k}\ge3$
are fixed positive integers and weakly optimal $K_{s_{i}}$-free pseudorandom
graphs exist for each $1\le i\le k$, then the multicolor Ramsey numbers
satisfy
\[
\Omega\Big(\frac{t^{S+1}}{\log^{2S}t}\Big)\le r(s_{1},\ldots,s_{k},t)\le O\Big(\frac{t^{S+1}}{\log^{S}t}\Big),
\]
as $t\rightarrow\infty$, where $S=\sum_{i=1}^{k}(s_{i}-2)$. This generalizes previous results of Mubayi and Verstra\"ete, who proved the case $k=1$, and Alon and R\"odl, who proved the case $s_1=\cdots = s_k = 3$. Both previous results used the existence of optimal rather than weakly optimal $K_{s_i}$-free graphs.
\end{abstract}

\section{Introduction}
The central object of study in Ramsey theory is the Ramsey number
$r(s_{1},\ldots,s_{k})$, which is defined to be the smallest posititive
integer $N$ such that in any $k$-coloring of the complete graph
$K_{N}$, there is a monochromatic $K_{s_{i}}$ of some color $i\in \{1,\ldots ,k\}$.

In the case $k=2$, the order of growth of $r(3,t)$ as $t\rightarrow\infty$
was determined to be
\[
r(3,t)=\Theta\Big(\frac{t^{2}}{\log t}\Big)
\]
by Ajtai, Koml\'os, and Szemer\'edi \cite{AjKoSz} and Kim \cite{Kim}.
It is one of the central open problems in Ramsey theory to generalize
these bounds and determine the growth rates of $r(s,t)$ for all fixed
$s\ge3$ and $t\rightarrow\infty$. Unfortunately, when $s\ge4$ even
the polynomial order of $r(s,t)$ is not known, and the best known
bounds are
\[
\Omega\Big(\frac{t^{\frac{s+1}{2}}}{(\log t)^{\frac{s+1}{2}-\frac{1}{s-2}}}\Big)\le r(s,t)\le O\Big(\frac{t^{s-1}}{\log^{s-2}t}\Big).
\]
The lower bound is due to Bohman and Keevash \cite{BoKe}, while the
upper bound is again due to Ajtai, Koml\'os, and Szemer\'edi \cite{AjKoSz}.

Recently, Mubayi and Verstra\"ete \cite{MuVe} connected the growth rate
of $r(s,t)$ to a problem in the theory of pseudorandom
graphs. Recall that an {\it $(n,d,\lambda)$-graph} is a $d$-regular graph on $n$ vertices such that all of its nontrivial eigenvalues have absolute value at most $\lambda$.
\begin{definition}
A {\it family of weakly optimal $K_{s}$-free $(n,d,\lambda)$-graphs} is a collection of $K_{s}$-free
$(n_i,d_i,\lambda_i)$-graphs for which $d_i=\Theta(n_i^{1-\alpha})$ and
$\lambda_i =\Theta(n_i^{1-(s-1)\alpha})$ as $n_i\rightarrow\infty$, for some fixed $\alpha>0$. We call $\alpha$ the {\it parameter of weak optimality}. If, moreover, $\lambda_i=\Theta(\sqrt{d_i})$ (so that $\alpha=\frac 1{2s-3}$), then this family is said to be {\it optimal.}
\end{definition}

Note that $\alpha$ and the implicit constants may not depend on $i$. Informally, we say that weakly optimal $K_{s}$-free $(n,d,\lambda)$-graphs exist if there exists a family of weakly optimal $K_{s}$-free $(n,d,\lambda)$-graphs, for some fixed $\alpha>0$. Note that the $t$-blowup of an $(n,d,\lambda)$-graph is an $(nt,dt,\lambda t)$-graph with the same clique number; thus, the existence of optimal $K_s$-free $(n,d,\lambda)$-graphs implies the existence of weakly optimal $K_s$-free $(n,d,\lambda)$-graphs for all $0 < \alpha \le \frac{1}{2s-3}$ (this fact was observed already by Krivelevich, Sudakov, and Szab\'o \cite{KrSuSz} when $s=3$). Because of this, the existence of weakly optimal $K_s$-free $(n,d,\lambda)$-graphs is indeed weaker than the existence of optimal ones. 

 Sudakov, Szab\'o, and Vu~\cite{SuSzVu} conjectured the existence of optimal $K_{s}$-free $(n,d,\lambda)$-graphs for all $s\ge 3$ and all $n$; such graphs where constructed by Alon~\cite{Al} in the case $s=3$ but the conjecture remains open for $s\ge 4$ (see \cite{BiIhPe} for the best known construction for $s\ge5$, which agrees with Alon's bound for $s=4$). Conditional on this conjecture, Mubayi and Verstra\"ete showed that $r(s,t)$
grows like $t^{s-1}$ up to polylogarithmic factors.
\begin{theorem}
\label{thm:mubayi-verstraete}(Mubayi and Verstra\"ete \cite{MuVe}.) If optimal $K_{s}$-free
$(n,d,\lambda)$-graphs exist for all $n$, then
\[
\Omega\Big(\frac{t^{s-1}}{\log^{2s-4}t}\Big)\le r(s,t)\le O\Big(\frac{t^{s-1}}{\log^{s-2}t}\Big),
\]
where the implicit constants may depend only on $s$.
\end{theorem}

Theorem~\ref{thm:mubayi-verstraete} relies heavily on a lemma of
Alon and R\"odl \cite{AlRo}, which was originally used to prove the following bound on the multicolor Ramsey number $r_{k}(s,t)\coloneqq r(s,\ldots,s,t)$
where $s$ appears $k$ times.
\begin{theorem}
\label{thm:alon-rodl}(Alon and R\"odl \cite{AlRo}.) For all $k\ge1$,
\[
\Omega\Big(\frac{t^{k+1}}{\log^{2k}t}\Big)\le r_{k}(3,t)\le O\Big(\frac{t^{k+1}}{\log^{k}t}\Big),
\]
where the implicit constants may depend only on $k$.
\end{theorem}

Note that Theorem~\ref{thm:alon-rodl} depends on the existence of optimal $K_{3}$-free $(n,d,\lambda)$-graphs, which were constructed by Alon \cite{Al}. 

Our main result is the following natural common generalization of Theorems~\ref{thm:mubayi-verstraete}
and~\ref{thm:alon-rodl}, which also replaces the assumption of optimality by that of weak optimality.
\begin{theorem}
\label{thm:main}If $s_{1},\ldots,s_{k}\ge3$, $S=\sum_{i=1}^{k}(s_{i}-2)$,
and for each $1\le i\le k$ there exist weakly optimal $K_{s_{i}}$-free
$(n,d,\lambda)$-graphs for all $n$, then
\begin{equation}
\Omega\Big(\frac{t^{S+1}}{\log^{2S}t}\Big)\le r(s_{1},\ldots,s_{k},t)\le O\Big(\frac{t^{S+1}}{\log^{S}t}\Big),\label{eq:main}
\end{equation}
where the implicit constants may depend only on $S$ and the weak optimality parameters $\alpha_1,\ldots,\alpha_k$.
\end{theorem}

Like Theorems~\ref{thm:mubayi-verstraete} and~\ref{thm:alon-rodl},
Theorem~\ref{thm:main} is a consequence of a lemma of Alon and R\"odl
\cite{AlRo} which shows that an $(n,d,\lambda)$-graph has few independent
sets of order just over $n/d$. We will need the following slightly stronger version,
which is proved in exactly the same way.
\begin{lemma}
\label{lem:alon-rodl}If $G$ is an $(n,d,\lambda)$-graph and $t\ge\frac{2n\log^{2}n}{d}$,
then the number of $t$-tuples $(v_{1},\ldots,v_{t})\in V(G)^t$ of vertices of $G$, no pair of which are adjacent, is at most
\[
\Big(\frac{4en\lambda}{d}\Big)^{t}.
\]
\end{lemma}

In the next section we prove the lower bound in Theorem~\ref{thm:main}. The proofs of Lemma~\ref{lem:alon-rodl} and the upper bound in Theorem~\ref{thm:main} are relatively standard and are confined to the appendix.

\section{The Proof}

The main difficulty in applying Lemma~\ref{lem:alon-rodl} to construct Ramsey graphs is rescaling a given $(n,d,\lambda)$-graph to have the
appropriate number of vertices. The proofs of Theorems~\ref{thm:mubayi-verstraete}
and~\ref{thm:alon-rodl} each provide half the picture. In the proof of Theorem
\ref{thm:mubayi-verstraete}, a $K_{s}$-free $(n,d,\lambda)$-graph
is scaled down to a smaller $K_s$-free graph with no independent sets of size
$t$ by sampling a random induced subgraph. In the proof of Theorem~\ref{thm:alon-rodl},
a $K_{3}$-free $(n,d,\lambda)$-graph is scaled up to a larger $K_3$-free graph
with few independent sets by performing a balanced blowup. 

The natural common generalization of these two constructions is a random blowup; using random blowups, we will be able to scale the weakly optimal $K_{s}$-free $(n,d,\lambda)$-graphs
to $K_{s}$-free graphs of any size with few independent sets.
Define $i_{t}(G)$ to be the number of independent sets of order $t$
in $G$.
\begin{lemma}
\label{lem:main}If there exists a $K_{s}$-free $(n,d,\lambda)$-graph
$G$ and $t\ge\frac{2n\log^{2}n}{d}$, then for every $N$ there exists
a $K_{s}$-free graph $G(N)$ on $N$ vertices with
\[
i_{t}(G(N))\le\Big(\frac{2e^{2}\lambda N}{n\log^{2}n}\Big)^{t}.
\]
\end{lemma}

\begin{proof}
We will define $G(N)$ as follows. Pick a uniform random map $f:[N]\rightarrow G$,
and let $G(N)$ be the graph on $[N]$ whose edges are exactly the
pairs $(i,j)$ that map to edges in $G$. Since $G$ is $K_{s}$-free,
so is $G(N)$. It suffices to prove the desired upper bound on $\mathbb{E}[i_{t}(G(N))]$.

By Lemma~\ref{lem:alon-rodl} (proved in Appendix~\ref{apx:alon-rodl}) and linearity of expectation, 
\begin{eqnarray*}
\mathbb{E}[i_{t}(G(N))] & = & \binom{N}{t}\Pr[f([t])\text{ is an independent set}]\\
 & = & \binom{N}{t}\frac{\big(\frac{4e\lambda n}{d}\big)^{t}}{n^{t}},
\end{eqnarray*}
since $f([t])$ is a uniform random $t$-tuple in $V(G)^{t}$. Bounding
$\binom{N}{t}\le\big(\frac{eN}{t}\big)^{t}$, we find that with positive
probability,
\[
i_{t}(G(N))\le\Big(\frac{eN}{t}\Big)^{t}\Big(\frac{4e\lambda}{d}\Big)^{t}\le\Big(\frac{2e^{2}\lambda N}{n\log^{2}n}\Big)^{t}
\]
since $t\ge\frac{2n\log^{2}n}{d}$.
\end{proof}
We are ready to prove the main result. The upper bound is proved in Appendix~\ref{apx:main-upper}.
\begin{proof}[Proof of the lower bound in Theorem~\ref{thm:main}.] 
Henceforth all implicit constants are allowed to depend on $S=\sum_{i=1}^k(s_i-2)$ and on the weak optimality parameters $\alpha_1,\ldots,\alpha_k$. Let $G_{i}$ be a weakly
optimal $K_{s_{i}}$-free $(n_{i},d_{i},\lambda_{i})$-graph, where $d_i=\Theta(n_i^{1-\alpha_i})$ and $\lambda_i=\Theta(n_i^{1-(s_i-1)\alpha_i})$. As these are assumed to exist for all $n_i$, we pick 
\[
n_{i}=\Theta\Big(\Big(\frac{t}{\log^{2}t}\Big)^{1/\alpha_i}\Big)
\]
so that with $d_{i}=\Theta(n_{i}^{1-\alpha_i})$,
the bound $t \ge \frac{2n_{i}\log^{2}n_{i}}{d_{i}}$ holds.
Take
\[
N=\Theta\Big(\frac{t^{S+1}}{\log^{2S}t}\Big),
\]
the implicit constant to be chosen later. Rescaling each $G_{i}$ to
a $G_{i}(N)$ on $N$ vertices satisfying Lemma~\ref{lem:main}, we
get $k$ graphs $G_{i}(N)$ on the same vertex set $[N]$ such that $G_{i}(N)$ is $K_{s_{i}}$-free
and
\begin{equation}
i_{t}(G_{i}(N))\le\Big(\frac{2e^{2}\lambda_{i}N}{n_{i}\log^{2}n_{i}}\Big)^{t}.\label{eq:ind-sets}
\end{equation}
We define a random $(k+1)$-coloring of $\binom{[N]}{2}$ so that in each of
the first $k$ colors, the edges form a subgraph of $G_{i}(N)$. To
do so, simply take a uniform random vertex permutation of $G_{i}(N)$
as the edges in the $i$-th color; when multiple colors are given
to the same edge, break ties arbitrarily. All remaining edges are given color $k+1$. 

This $(k+1)$-colored graph has
no monochromatic $K_{s_{i}}$ in any of the first $k$ colors. It
remains to show that with positive probability, it has no $K_{t}$ in
the last color. Indeed, the probability that a given set $I$ of order
$t$ induces a $K_{t}$ in the last color is exactly the product
\[
\prod_{i=1}^{k}\frac{i_{t}(G_{i}(N))}{\binom{N}{t}},
\]
since $I$ must be an independent set in each of the first $k$ colors. By (\ref{eq:ind-sets}), we have that
\begin{eqnarray*}
\prod_{i=1}^{k}\frac{i_{t}(G_{i}(N))}{\binom{N}{t}} & \le & \prod_{i=1}^{k}\Big(\frac{2e^{2}\lambda_{i}N}{n_{i}\log^{2}n_{i}}\Big)^{t}/\Big(\frac{N}{t}\Big)^{t}\\
 & \le & \prod_{i=1}^{k}(C\lambda_{i}/d_{i})^{t}
\end{eqnarray*}
for an absolute constant $C>0$. With our choices of $\lambda_{i}$
and $d_{i}$,
\[
\frac{\lambda_{i}}{d_{i}}=\Theta \left( n_i^{-\alpha_i(s_i-2)} \right)=\Theta\left( \left(\frac t{\log^2 t}\right)^{-(s_i-2)} \right).
\]
By taking a union bound over all $I$, the probability that there exists a $K_{t}$
in the last color is at most
\[
\binom{N}{t}\prod_{i=1}^{k}O\Big(\Big(\frac{t}{\log^{2}t}\Big)^{-(s_{i}-2)}\Big)^{t}\le O\Big(\frac{N}{t}\Big(\frac{t}{\log^{2}t}\Big)^{-S}\Big)^{t}<1
\]
for the appropriate choice of the constant in the definition of $N$.
This completes the proof.
\end{proof}

\noindent {\bf Acknowledgements.} The authors would like to thank Ryan Alweiss and Jacob Fox for helpful discussions on this problem. We are also grateful to Anurag Bishnoi for bringing reference \cite{KrSuSz} to our attention.

\appendix
\section{Proof of Lemma~\ref{lem:alon-rodl}}\label{apx:alon-rodl}
We give a short proof of Lemma~\ref{lem:alon-rodl} using the Expander Mixing Lemma (see e.g. \cite[Corollary 9.2.5]{AlSp}).
\begin{lemma}
\label{lem:expander-mixing}(Expander Mixing Lemma.) If $G$ is an
$(n,d,\lambda)$-graph and $S,T\subseteq V(G)$, then
\[
|e(S,T)-\frac{d}{n}|S||T||<\lambda\sqrt{|S||T|}.
\]
\end{lemma}
\noindent Here $e(S,T)$ denotes the number of ordered pairs $(s,t)\in S\times T$
which are edges of $G$.
\begin{proof}[Proof of Lemma~\ref{lem:alon-rodl}.]
 We count the number of ways to pick $v_{1},\ldots,v_{t}$ one-by-one.
Let $S_{k}$ be the set of all vertices with no edges to $v_{1},\ldots,v_{k-1}$
(including $v_{1},\ldots,v_{k-1}$), and let $T_{k}=\{v\in S_{k}:|N(v)\cap S_{k}|<\frac{d}{2n}|S_{k}|\}$.
Thus, $S_{k}$ is the set of all valid candidates for $v_{k}$, and
$T_{k}$ is the subset of valid candidates for which $S_{k+1}$ is
not much smaller than $S_{k}$. In particular, every time we choose
$v_{k}\in S_{k}\backslash T_{k}$, we find that
\[
|S_{k+1}|\le(1-\frac{d}{2n})|S_{k}|< e^{-\frac{d}{2n}}|S_{k}|,
\]
so since $|S_{0}|=n$, the total number of $k$ for which $v_{k}$
can be chosen from $S_{k}\backslash T_{k}$ is bounded by $t'=\frac{2n}{d}\log n$.

On the other hand, by the definition of $T_{k}$ we have $e(S_{k},T_{k})<\frac{d}{2n}|S_{k}||T_{k}|$,
and so applying Lemma~\ref{lem:expander-mixing} we get
\[
\frac{d}{2n}|S_{k}||T_{k}|<\lambda\sqrt{|S_{k}||T_{k}|}.
\]
In particular, since $T_{k}\subseteq S_{k}$, we have
\[
|T_{k}|<\frac{2n\lambda}{d}.
\]

Thus, the total number of sequences $v_{1},\ldots,v_{t}$ where all
pairs are not adjacent is bounded by
\[
\binom{t}{t'}n^{t'}\Big(\frac{2n\lambda}{d}\Big)^{t},
\]
since we can choose the $t'$ steps on which $v_{k}\in S_{k}\backslash T_{k}$
in $\binom{t}{t'}$ ways, the number of such choices is bounded by
$n$ on each step, and in all the other steps the number of choices
for $v_{k}$ is at most $|T_{k}|<\frac{2n\lambda}{d}$. Bounding $\binom{t}{t'}<2^{t}$
and $n^{t'}<n^{t/\log n}=e^{t}$, we obtain a bound of
\[
\Big(\frac{4en\lambda}{d}\Big)^{t},
\]
as claimed.
\end{proof}

\section{The upper bound in Theorem~\ref{thm:main}}\label{apx:main-upper}
 Alon and R\"odl \cite{AlRo} proved the upper bound in (\ref{eq:main})
when $s_{1}=s_{2}=\cdots=s_{k}=3$, and our proof is a generalization
of theirs.
\begin{proof}[Proof of the upper bound in Theorem~\ref{thm:main}.]
We fix $k$ and induct on $S$. The base case $S=1$ is just $r(2,2,\ldots,2,3,t)=O(t^{2}/\log t)$
for any number of $2$'s, by Ajtai, Koml\'os and Szemer\'edi~\cite{AjKoSz}. Assume by induction that there exist absolute
constants $C_{S'}>0$ for all $S'<S$ such that for all vectors $(s_{1},\ldots,s_{k})$
with $s_{i}\ge2$ and $\sum_{i=1}^{k}(s_{i}-2)=S'$,
\[
r(s_{1},\ldots,s_{k},t)\le n_{S'}\coloneqq\frac{C_{S'}t^{S'+1}}{\log^{S'}t}.
\]
Now let $n_{S}=C_{S}t^{S+1}/\log^{S}t$ for some $C_{S}$ to be determined,
and suppose we are given a $(k+1)$-coloring of $K_{n_{S}}$ such
that there is no monochromatic $K_{s_{i}}$ of color $i$, nor a monochromatic
$K_{t}$ of color $k+1$. Define $T$ to be the spanning subgraph
of $K_{n_{S}}$ obtained by taking only the edges of the first $k$
colors. If $D$ is the maximum degree in $T$, then
\begin{equation}
D<kn_{S-1},\label{eq:D}
\end{equation}
If (\ref{eq:D}) is false, then there is a vertex $v\in V(T)$ and
some color $i\le k$ such that $v$ is incident to at least
\[
n_{S-1}\ge r(s_{1},\ldots,s_{i}-1,\ldots,s_{k},t)
\]
edges of color $i$. The induced subgraph on the set of vertices connected
to $v$ by color $i$ must not contain a monochromatic clique $K_{s_{j}}$
of any color $j\ne i$, so there will be a $K_{s_{i}-1}$ of color
$i$ inside. But then this forms a $K_{s_{i}}$ of color $i$ together
with $v$, which is a contradiction. This proves inequality (\ref{eq:D}).

Next, let $D'$ denote the maximum number of edges in some neighborhood
$N_T(v)$ of a vertex in $T$. We show
\begin{equation}
D'<k^{2}Dn_{S-2}.\label{eq:D'}
\end{equation}
Suppose otherwise, and let $v$ be the vertex with the most edges
in its neighborhood. If $u\in N_T(v)$, define $d_{v}(u)$ as the number
of common neighbors $w\in N_T(v)\cap N_T(u)$ for which either $uv,uw,vw$
are all the same color, or $uw$ and $vw$ are different colors. Each
edge $uw\in N_T(v)$ contributes either once or twice to the sum of
the $d_{v}(u)$, so
\[
\sum_{u\in N_T(v)}d_{v}(u)\ge k^{2}Dn_{S-2}.
\]

In particular, there is some $u$ for which $d_{v}(u)\ge k^{2}n_{S-2}.$
We can categorize the vertices $w$ of $N_T(v)$ counted in $d_{v}(u)$
by the pair of colors of $uw$ and $vw$, and find that there exists
colors $i,j$ (not necessarily different) and a set $W$ of $n_{S-2}$
vertices such that for every $w\in W$, $uw$ is of color $i$ and
$vw$ is of color $j$. If $i\ne j$, this implies a contradiction
from the fact that
\[
|W|\ge n_{S-2}\ge r(s_{1},\ldots,s_{i}-1,\ldots,s_{j}-1,\ldots,s_{k},t).
\]
Otherwise, if $i=j$, then by the definition of $d_{v}(u)$ it must
be that $uv$ is of color $i$ as well, and so we also get a contradiction
since
\[
|W|\ge n_{S-2}\ge r(s_{1},\ldots,s_{i}-2,\ldots,s_{k},t).
\]

This proves (\ref{eq:D'}). It is a corollary of a result of Alon,
Krivelevich, and Sudakov \cite{AlKrSu} that if a graph has maximum
degree $D$ and every neighborhood has at most $D'=\frac{D^{2}}{f}$
edges, then its independence number is at least $\Omega(\frac{n\log f}{D})$.
In particular, we see that the independence number of $T$ is at least
\[
\Omega\Big(\frac{n_{S}\log t}{D}\Big),
\]
since (\ref{eq:D'}) implies $D'=O(D^{2}\log t/t)$. On the other
hand, an independent set in $T$ forms a monochromatic clique in $K_{n_{S}}$
of color $k+1$, so 
\[
t>\Omega\Big(\frac{n_{S}\log t}{D}\Big),
\]
which shows that 
\[
n_{S}<O\Big(\frac{Dt}{\log t}\Big)=O\Big(\frac{C_{S-1}t^{S+1}}{\log^{S}t}\Big).
\]
 Picking $C_{S}$ sufficiently large in terms of $C_{S-1}$, this
gives the desired contradiction.
\end{proof}
\end{document}